\documentclass{amsart}

\usepackage{amssymb,mathtools}
\usepackage{enumitem}
\usepackage{hyperref}

\IfFileExists{bbm.sty}{\usepackage{bbm}\newcommand{\one}{\mathbbm{1}}}{\newcommand{\one}{\mathbf{1}}}

\newcommand{\C}{\mathcal{C}}
\newcommand{\D}{\mathcal{D}}
\newcommand{\E}{\mathcal{E}}
\newcommand{\F}{\mathcal{F}}
\newcommand{\I}{\mathcal{I}}
\newcommand{\Pc}{\mathcal{P}}
\newcommand{\Oc}{\mathcal{O}}
\newcommand{\Zc}{\mathcal{Z}}
\newcommand{\QQ}{\mathbb{Q}}
\newcommand{\ZZ}{\mathbb{Z}}

\newcommand{\CC}{\mathbb{C}}
\newcommand{\Qbar}{\overline{\mathbb{Q}}}
\newcommand{\FSexp}{\operatorname{FSexp}}
\newcommand{\Ndim}{\operatorname{Ndim}}
\newcommand{\FPdim}{\operatorname{FPdim}}
\newcommand{\Tr}{\operatorname{Tr}}
\newcommand{\Gal}{\operatorname{Gal}}
\newcommand{\ord}{\operatorname{ord}}
\newcommand{\lcm}{\operatorname{lcm}}
\newcommand{\rank}{\operatorname{rank}}
\newcommand{\Stab}{\operatorname{Stab}}
\newcommand{\Hom}{\operatorname{Hom}}
\newcommand{\Rep}{\operatorname{Rep}}
\newcommand{\VecG}[2]{\mathrm{Vec}_{#1}^{#2}}
\newcommand{\pt}{\mathrm{pt}}
\newcommand{\ad}{\mathrm{ad}}

\newcommand{\Av}{\mathcal{A}}
\newcommand{\sltwo}{\mathfrak{sl}_2}
\newcommand{\sofive}{\mathfrak{so}_5}
\newcommand{\hs}{\hat{\sigma}}

\newtheorem{theorem}{Theorem}[section]
\newtheorem{lemma}[theorem]{Lemma}
\newtheorem{proposition}[theorem]{Proposition}
\newtheorem{corollary}[theorem]{Corollary}
\newtheorem*{theorem*}{Theorem}
\newtheorem*{corollary*}{Corollary}
\theoremstyle{definition}
\newtheorem{example}[theorem]{Example}
\theoremstyle{remark}
\newtheorem{remark}[theorem]{Remark}

\numberwithin{equation}{section}

\setlist[enumerate]{leftmargin=2.2em,font=\upshape}
\setlist[enumerate,1]{label=(\alph*)}
\setlist[enumerate,2]{label=(\arabic*),ref=(\alph{enumi})(\arabic*)}

\begin{document}

\title{On Frobenius--Schur exponent bounds}

\author{Agustina Czenky}
\address{Department of Mathematics, University of Southern California, Los Angeles, CA 90089, USA}
\email{czenky@usc.edu}

\author{Julia Plavnik}
\address{Department of Mathematics, Indiana University, Bloomington, IN 47405, USA}
\email{jplavnik@iu.edu}
\thanks{The second author was partially supported by NSF grant DMS-2146392 and by Simons Foundation Award 889000 as part of the Simons Collaboration on Global Categorical Symmetries.}

\author{Andrew Schopieray}
\address{Department of Mathematics, SUNY Plattsburgh, Plattsburgh, NY 12901, USA}
\email{ascho013@plattsburgh.edu}


\begin{abstract}
We study bounds on the Frobenius--Schur exponent of spherical fusion categories in terms of their global dimension, generalizing bounds from the representation theory of finite-dimensional quasi-Hopf algebras. Our main result is that if the Frobenius--Schur exponent of a modular fusion category is a power of a prime $p$, then it is bounded by the norm of its global dimension when $p$ is odd, and by four times the norm of its global dimension when $p=2$; these bounds are optimal. Moreover, it is bounded by a constant multiple of the geometric mean of the Galois conjugates of its global dimension, which is much smaller when the global dimension is irrational, and a bound of this type is proven for each prime power dividing the Frobenius--Schur exponent of an arbitrary modular or spherical fusion category. Under pseudounitarity, we describe all modular and spherical fusion categories whose Frobenius--Schur exponent is a prime power and at least the norm (respectively, the square of the norm) of their global dimension, using a classification of modular fusion categories of Frobenius--Perron dimension $p^n$ for $n\leq5$.
\end{abstract}

\maketitle

\section{Introduction}

The characters of a finite group $G$ take values in the cyclotomic field $\QQ(\zeta_n)$, where $\zeta_n$ is a primitive $n$th root of unity and $n:=|G|$. Analogous statements can be made about doubles of finite groups, and more generally about fusion and modular fusion categories over $\CC$, whose characters are controlled by their modular data, i.e.\ their $S$- and $T$-matrices \cite[Section 8.13]{EGNO}. The $T$-matrix of a modular fusion category $\C$ has finite order $N$, equal to the Frobenius--Schur exponent $\FSexp(\C)$ \cite[Theorem 7.7]{NS1}, and all modular data of $\C$ lie in $\QQ(\zeta_N)$ \cite[Theorem 8.14.7]{EGNO}. Our goal is to bound $\FSexp(\C)$ by the global dimension of $\C$. The Frobenius--Schur exponent of a spherical fusion category equals that of its Drinfeld center \cite[Corollary 7.8]{NS1}, so this goal extends to all spherical fusion categories.

The Frobenius--Schur exponent of a spherical fusion category $\C$ is intimately related to $\Ndim(\C)$, the algebraic norm of its global dimension $\dim(\C)$, i.e.\ the product of the distinct Galois conjugates of $\dim(\C)$. In particular, $\FSexp(\C)$ and $\Ndim(\C)$ have the same prime factors \cite[Theorem 3.9]{BNRW}. \cite[Proposition 4.2]{Sch-norm} and \cite[Theorem 5.1]{Etingof} give a rudimentary bound: for any modular fusion category $\C$, $\FSexp(\C)^{2d}$ divides $\Ndim(\C)^5$, where $d:=[\QQ(\dim(\C)):\QQ]$ (take norms from $\QQ(\dim(\C)^{1/2})$ in the divisibility $\FSexp(\C)\mid\dim(\C)^{5/2}$ of algebraic integers). Passing to the Drinfeld center \cite[Section 7.13]{EGNO}, for an arbitrary spherical fusion category $\C$, $\FSexp(\C)^{d}$ divides $\Ndim(\C)^5$. Even when $d=1$, these results allow $\FSexp(\C)$ to be as large as $\Ndim(\C)^{5/2}$ (respectively, $\Ndim(\C)^5$); one motivation for this paper is that we do not believe such categories exist.

In the classical case where $\C$ is the representation category of a semisimple quasi-Hopf algebra $H$, $\C$ is integral, $\Ndim(\C)=\dim(H)$, and there are explicit bounds: $\FSexp(\C)\leq\Ndim(\C)^4$ \cite[Theorem 9.1]{NS1}; $\FSexp(\C)\leq\Ndim(\C)^3$ if $H$ is a Hopf algebra \cite{EG-exp}; and $\FSexp(\C)\leq\Ndim(\C)^2$ if $\C$ is group-theoretical \cite[Theorem 9.2]{NS1}. Here we do not assume that $\FPdim(\C)$ is an integer, which takes us outside the world of (quasi-)Hopf algebras: for instance, for an odd prime $p$ with $p^n\geq5$, the categories $\C(\sltwo,p^n,q)_\ad$ (Example \ref{ex:qg}) are Galois conjugates of pseudounitary categories with prime power Frobenius--Schur exponent, but $\FPdim(\C(\sltwo,p^n,q)_\ad)=p^n/(4\sin^2(\pi/p^n))\notin\ZZ$.

Our main result is the following optimal bound (Theorem \ref{thm:main}, Proposition \ref{prop:general}).

\begin{theorem*}
Let $\C$ be a modular fusion category. If $\FSexp(\C)$ is a power of a prime $p$, then $\FSexp(\C)\leq\Ndim(\C)$ if $p$ is odd and $\FSexp(\C)\leq4\cdot\Ndim(\C)$ if $p=2$. These bounds are optimal for all $p$. Moreover, $\FSexp(\C)\leq c_p\cdot\Ndim(\C)^{1/d}$, where $d:=[\QQ(\dim(\C)):\QQ]$, $c_p:=2p/(p-1)$ for odd $p$, and $c_2:=8$.
\end{theorem*}

Here $\Ndim(\C)^{1/d}$ is the geometric mean of the Galois conjugates of $\dim(\C)$. The last bound does not require a prime power: for any modular fusion category, each prime power $q$ exactly dividing $\FSexp(\C)$ is at most $c_p\cdot\Ndim(\C)^{1/d}$, so $\FSexp(\C)$ is bounded by a constant multiple of $\Ndim(\C)^{r/d}$, where $r$ is the number of primes dividing $\Ndim(\C)$ (Proposition \ref{prop:general}). Passing to the Drinfeld center, the same holds for all spherical fusion categories with $\Ndim(\C)$ replaced by $\Ndim(\C)^2$; in particular, $\FSexp(\C)\leq\Ndim(\C)^2$ ($p$ odd) and $\FSexp(\C)\leq4\cdot\Ndim(\C)^2$ ($p=2$) whenever $\FSexp(\C)$ is a power of $p$, without assuming pseudounitarity (Corollary \ref{cor:general}).

With the assumption of pseudounitarity, the extremal cases can be described completely (Corollary \ref{cor:modular}).

\begin{corollary*}
Let $\C$ be Galois conjugate to a pseudounitary modular fusion category. If $\FSexp(\C)$ is a prime power, then $\FSexp(\C)<\Ndim(\C)$ unless $\C$ is pointed, $\C$ is a Fibonacci modular fusion category, $\C$ is a product of two Ising modular fusion categories, or $\C$ is the product of an Ising modular fusion category and a pointed modular fusion category of dimension $1$, $2$, or $2^2$.
\end{corollary*}

Passing to Drinfeld centers, this implies the following (Corollary \ref{cor:spherical}).

\begin{corollary*}
Let $\C$ be Galois conjugate to a pseudounitary fusion category. If $\FSexp(\C)$ is a prime power, then $\FSexp(\C)\leq\Ndim(\C)^2$, with equality if and only if $\C\simeq\VecG{G}{\omega}$ for a cyclic $p$-group $G$ and a generator $\omega\in H^3(G,\CC^\times)$, or $\C$ is an Ising fusion category.
\end{corollary*}

Further consequences are collected in Section \ref{sec:FSfurther}: the orders of the twists of a simple object are bounded by its formal codegree when the global dimension is an integer, the bounds extend to the $T$-matrices of premodular fusion categories, and every nilpotent modular fusion category satisfies the Kashina-type divisibility $\FSexp(\C)\mid4\cdot\dim(\C)$, with $\FSexp(\C)\mid\dim(\C)$ outside three explicit families.

The proofs rest on a counting lemma for Galois orbits of simple objects (Lemma \ref{lem:orbit}), which bounds the size of an orbit from below in terms of the normalized $T$-matrix and from above in terms of $\Ndim(\C)$. The case $p=2$ also uses an explicit classification of modular fusion categories of Frobenius--Perron dimension $p^n$ for $n\leq5$ (Proposition \ref{prop:classification}), which extends results of the first two authors \cite{CP}.

It would be of interest to extend the above to nonpseudounitary modular fusion categories. For instance, the rank $6$ modular fusion categories $\C(\sofive,9,q)_\ad$ have global dimension $3^2$ and Frobenius--Schur exponent $3^2$, but are emphatically nonpseudounitary \cite{Sch-nonpu} (see Example \ref{ex:qg}). It would be equally interesting to show that the categories $\C(\sofive,9,q)_\ad$ are the only modular fusion categories with prime power Frobenius--Schur exponent and $\FSexp(\C)\geq\Ndim(\C)$ aside from those obtained from the categories listed above by Galois conjugation and by changing the spherical structure (e.g.\ the semion category with its nonpseudounitary spherical structure has $\FSexp(\C)=4>2=\Ndim(\C)$), or to find them in an infinite family.

\section{Preliminaries and notation}

\subsection{Fusion and modular fusion categories}\label{sec:fusion}

We follow \cite{EGNO} for fusion, braided, and spherical fusion categories. For a fusion category $\C$, let $\Oc(\C)$ be the set of isomorphism classes of simple objects, $\one$ the unit, $\C_\pt$ the maximal pointed fusion subcategory, and $\C_\ad$ the adjoint subcategory, generated by the $X\otimes X^\ast$, $X\in\Oc(\C)$; braidings are denoted $c_{X,Y}$. A spherical structure gives categorical dimensions $\dim(X)=\dim(X^\ast)$, which form a real character of the Grothendieck ring \cite[Proposition 4.7.12]{EGNO}, possibly different from $\FPdim$; for a set $\mathcal{S}$ of simple objects let $\dim(\mathcal{S}):=\sum_{X\in\mathcal{S}}\dim(X)^2$. The global dimension $\dim(\C):=\sum_{X\in\Oc(\C)}\dim(X)^2$ is a totally positive algebraic integer, each of whose Galois conjugates is the global dimension of a Galois conjugate fusion category (Section \ref{sec:galois}), hence is at least $1$. We write $\Ndim(\C)\in\ZZ_{\geq1}$ for its algebraic norm, the product of its distinct Galois conjugates. When $\dim(\C)=\FPdim(\C):=\sum_{X\in\Oc(\C)}\FPdim(X)^2$, a condition called \emph{pseudounitarity}, there is a unique spherical structure with $\dim=\FPdim$ as characters, and weakly integral fusion categories, i.e.\ those with $\FPdim(\C)\in\ZZ$, are pseudounitary \cite[Propositions 8.23, 8.24]{ENO}; we always equip pseudounitary fusion categories with this spherical structure. Weakly integral fusion categories admit a faithful grading by an elementary abelian $2$-group whose trivial component consists of the objects $X$ with $\FPdim(X)\in\ZZ$ \cite[Theorem 3.10]{GN}; in particular, a weakly integral fusion category of odd Frobenius--Perron dimension is integral.

On a braided fusion category, spherical structures correspond to ribbon structures $\theta$ via $\theta_X\dim(X)=\Tr(c_{X,X})$ \cite[Section 8.10]{EGNO}; for simple $X$ we regard $\theta_X=\theta_{X^\ast}$ as a scalar. Spherical braided fusion categories are called \emph{premodular}, and their $S$- and $T$-matrices are $S_{X,Y}:=\Tr(c_{Y,X}c_{X,Y})$ and $T_{X,Y}:=\delta_{X,Y}\theta_X$. We say $\C$ is a \emph{modular} fusion category if $S$ is nondegenerate. In this case the entries of $T$ are roots of unity \cite[Corollary 8.18.2]{EGNO}, $\ord(T)=\FSexp(\C)$ \cite[Theorem 7.7]{NS1}, and the entries of $S$ are cyclotomic integers in $\QQ(\zeta_N)$, $N:=\FSexp(\C)$ \cite[Theorem 8.14.7]{EGNO}; in particular $\dim(\C)\in\QQ(\zeta_N)^+$. Let $\gamma$ be a cube root of $\tau(\C)/\sqrt{\dim(\C)}$, where $\tau(\C):=\sum_{X\in\Oc(\C)}\dim(X)^2\theta_X$ and $\sqrt{\dim(\C)}$ is the positive root. Then $\dim(\C)^{-1/2}S$ and $t:=\gamma^{-1}T$ define a linear representation of $\mathrm{SL}(2,\ZZ)$ \cite{DLN}; we write $t_X:=\gamma^{-1}\theta_X$, so that $t_X=\theta_Xt_\one$.

For a spherical fusion category $\C$, the Drinfeld center $\Zc(\C)$ is a modular fusion category with $\dim(\Zc(\C))=\dim(\C)^2$ \cite[Section 7.13]{EGNO} and $\FSexp(\Zc(\C))=\FSexp(\C)$ \cite[Corollary 7.8]{NS1}; as $\dim(\C)$ is totally positive, $\QQ(\dim(\C)^2)=\QQ(\dim(\C))$, so $\Ndim(\Zc(\C))=\Ndim(\C)^2$.

We use the following standard facts about a modular fusion category $\C$ throughout. Simple objects $X,Y$ centralize one another, i.e.\ $c_{Y,X}c_{X,Y}=\mathrm{id}$, if and only if $S_{X,Y}=\dim(X)\dim(Y)$ \cite[Proposition 2.5]{Muger}. For a fusion subcategory $\D\subset\C$ with centralizer $C_\C(\D)$, we have $\dim(\D)\dim(C_\C(\D))=\dim(\C)$ \cite[Theorem 3.2]{Muger}, and if $\D$ is modular then $\C\simeq\D\boxtimes C_\C(\D)$ \cite[Theorem 4.2]{Muger}. The balancing equation reads $\theta_X\theta_YS_{X,Y}=\sum_{Z}N_{X,Y}^Z\theta_Z\dim(Z)$ \cite[Proposition 8.13.8]{EGNO}. For $g\in\Oc(\C_\pt)$ and $X\in\Oc(\C)$, the double braiding $c_{X,g}c_{g,X}$ is a scalar $\chi_X(g)$, and $X\mapsto\chi_X$ identifies the universal grading group of $\C$ with the character group of $\Oc(\C_\pt)$ \cite[Theorem 6.3]{GN}; moreover $C_\C(\C_\pt)=\C_\ad$ and $C_\C(\C_\ad)=\C_\pt$ \cite[Section 6]{GN}. All components of a faithful grading have the same dimension \cite[Theorem 3.5.2]{EGNO}. Note $\chi_{X^\ast}=\chi_X^{-1}$ and $\dim(g)=\pm1$. For noninvertible $X\in\Oc(\C)$, let $\Stab(X):=\{g\in\Oc(\C_\pt):g\otimes X\cong X\}$; this is the set of invertible summands of $X\otimes X^\ast\in\C_\ad$, so $\Stab(X)\subset\Oc((\C_\ad)_\pt)$. If $\C$ is pseudounitary and $g\in\Oc(\C_\pt)$, then $\theta_{g^k}=\theta_g^{k^2}$ and $b(g,h):=\chi_h(g)=\theta_{gh}\theta_g^{-1}\theta_h^{-1}$ is a bicharacter on $\Oc(\C_\pt)$ with $b(g,g)=\theta_g^2$ \cite[Section 8.4]{EGNO}; in particular $\theta_g^{2\ord(g)}=b(g^{\ord(g)},g)=1$. We write $\C(A,q)$ for the pointed modular fusion category with group of invertible objects $A$ and twist $\theta=q$, a nondegenerate quadratic form on $A$.

The following observation will be used repeatedly.

\begin{lemma}\label{lem:central}
Let $\C$ be a pseudounitary modular fusion category and $g\in\Oc(\C_\pt)\cap\Oc(\C_\ad)$ with $\theta_g=1$. If $g\otimes X\cong X$ for every noninvertible $X\in\Oc(\C)$, then $g\cong\one$.
\end{lemma}

\begin{proof}
As $g\in\C_\ad=C_\C(\C_\pt)$, $g$ centralizes all invertible objects. If $X$ is noninvertible, then $g\otimes X\cong X$ and the balancing equation give $S_{g,X}=\theta_X^{-1}\theta_X\dim(X)=\dim(g)\dim(X)$, so $g$ centralizes $X$. Hence $g$ lies in the M\"uger center of $\C$, which is trivial.
\end{proof}

\subsection{Galois action and number theory}\label{sec:galois}

For $\sigma\in\Gal(\Qbar/\QQ)$ and a fusion category $\C$, the Galois conjugate fusion category $\C^\sigma$ has the same fusion rules as $\C$, but all categorical structures are conjugated by $\sigma$ \cite[Sections 4.2--4.3]{DHW}. In particular $\dim(\C^\sigma)=\sigma(\dim(\C))$, $\Ndim(\C^\sigma)=\Ndim(\C)$, and $\FSexp(\C^\sigma)=\FSexp(\C)$. When $\C$ is modular, there is also a Galois action on $\Oc(\C)$: for every $\sigma\in\Gal(\Qbar/\QQ)$ there is a permutation $\hs$ of $\Oc(\C)$ such that
\begin{equation}\label{eq:galois}
\sigma\left(\frac{S_{X,Y}}{S_{\one,Y}}\right)=\frac{S_{X,\hs(Y)}}{S_{\one,\hs(Y)}}\qquad\text{for all }X,Y\in\Oc(\C),
\end{equation}
and $\sigma\mapsto\hs$ is an action factoring through $\Gal(\QQ(\zeta_N)/\QQ)$, $N=\FSexp(\C)$. Hence, by orthogonality of the columns of $S$,
\begin{equation}\label{eq:dimgalois}
\dim(\hs(X))^2=\frac{\dim(\C)}{\sigma(\dim(\C))}\,\sigma\!\left(\dim(X)^2\right),
\end{equation}
so if $\dim(\C)\in\ZZ$, then $\dim(\hs(X))^2=\sigma(\dim(X)^2)$ for all $X\in\Oc(\C)$. We denote the orbit of $X\in\Oc(\C)$ by $\Oc_X$. We record further properties of this action.
\begin{enumerate}[label=(G\arabic*)]
\item\label{G:t} $t_{\hs(X)}=\sigma^2(t_X)$ for all $X\in\Oc(\C)$ \cite[Theorem II]{DLN}.
\item\label{G:dual} $\hs(X^\ast)=\hs(X)^\ast$, by applying \eqref{eq:galois} to $S_{X,Y^\ast}=S_{X^\ast,Y}$.
\item\label{G:unit} $\hs(\one)=\one$ if and only if $\sigma$ fixes all $\dim(Y)$, by \eqref{eq:galois} with $Y=\one$. So $|\Oc_\one|=[\QQ(\dim(Y):Y\in\Oc(\C)):\QQ]$, and $\Oc_\one=\{\one\}$ if all $\dim(Y)$ are rational.
\item\label{G:grading} For $g\in\Oc(\C_\pt)$, \eqref{eq:galois} with $X=g$ gives $\chi_{\hs(Y)}(g)=\sigma(\chi_Y(g))$. Thus $\hs$ maps the graded component indexed by $\chi$ to the one indexed by $\sigma\circ\chi$, and the set of simple objects centralizing a given set of invertible objects is Galois stable.
\end{enumerate}

Lastly we introduce some notation from number theory. For $m\in\ZZ_{\geq1}$ let $s(m):=|\{a^2:a\in(\ZZ/m\ZZ)^\times\}|$, the number of distinct $\sigma^2(\zeta)$ for $\zeta$ a primitive $m$th root of unity and $\sigma\in\Gal(\QQ(\zeta)/\QQ)$. Since reduction $(\ZZ/m\ZZ)^\times\to(\ZZ/p^k\ZZ)^\times$ is surjective, $s(m)\geq s(p^k)$ whenever $p^k$ divides $m$. For $k\geq1$,
\begin{equation}\label{eq:s2}
s(p^k)=p^{k-1}(p-1)/2\ \text{ for odd primes }p,\qquad s(2^k)=\max(1,2^{k-3}).
\end{equation}
In particular, every power $q>1$ of a prime $p$ satisfies
\begin{equation}\label{eq:cp}
q\leq c_p\cdot s(q),\qquad\text{where }c_p:=\frac{2p}{p-1}\leq3\text{ for odd }p\text{ and }c_2:=8.
\end{equation}
For a totally positive algebraic number $\alpha$ let $\Av(\alpha):=[\QQ(\alpha):\QQ]^{-1}\Tr_{\QQ(\alpha)/\QQ}(\alpha)$ be the average of its Galois conjugates; by the inequality of arithmetic and geometric means, $\Av(\alpha)\geq N_{\QQ(\alpha)/\QQ}(\alpha)^{1/[\QQ(\alpha):\QQ]}$, which is at least $1$ if $\alpha$ is an algebraic integer. This applies to $\alpha=\dim(X)^2$, $X\in\Oc(\C)$, as $\dim(X)$ is a nonzero real cyclotomic integer.
For $X\in\Oc(\C)$ set $w(X):=\dim(X)^2/\dim(\C)$. Then $\sum_{X\in\Oc(\C)}w(X)=1$, $\sigma(w(X))=w(\hs(X))$ by \eqref{eq:dimgalois}, and $w(X)^{-1}$ is an algebraic integer \cite[Proposition 8.14.6]{EGNO}; so $w(X)$ is totally positive. The following lemma is our main tool.

\begin{lemma}\label{lem:orbit}
Let $\C$ be a modular fusion category, $X\in\Oc(\C)$, and $d:=[\QQ(\dim(\C)):\QQ]$. Set $\epsilon_X:=2$ if $X\not\cong X^\ast$ and $\epsilon_X:=1$ otherwise.
\begin{enumerate}
\item\label{orbit:lower} $|\Oc_X|\geq\epsilon_X\cdot s(\ord(t_X))$.
\item\label{orbit:upper} $\sum_{Y\in\Oc_X}w(Y)=|\Oc_X|\cdot\Av(w(X))\geq|\Oc_X|\cdot\Ndim(\C)^{-1/d}$. In particular $|\Oc_X|\leq\Ndim(\C)^{1/d}$.
\item\label{orbit:integer} If $\dim(\C)\in\ZZ$, then $\dim(\Oc_X)=|\Oc_X|\cdot\Av(\dim(X)^2)$, and $\dim(\Oc_X)<\dim(\C)$ unless $\C\simeq\mathrm{Vec}$.
\item\label{orbit:intdims} If $\dim(Y)\in\ZZ$ for all $Y\in\Oc(\C)$, then $\dim(\hs(X))=\pm\dim(X)$ and $\theta_{\hs(X)}=\sigma^2(\theta_X)$ for all $\sigma$, and $|\Oc_X|\geq\epsilon_X\cdot s(\ord(\theta_X))$.
\end{enumerate}
\end{lemma}

\begin{proof}
By \ref{G:t}, $\{t_Y:Y\in\Oc_X\}=\{\sigma^2(t_X):\sigma\in\Gal(\Qbar/\QQ)\}$ has $s(\ord(t_X))$ elements. If $X\not\cong X^\ast$, then by \ref{G:dual} every $Y\in\Oc_X$ satisfies $Y\not\cong Y^\ast\in\Oc_X$, and $t_{Y^\ast}=t_Y$, so every fiber of $Y\mapsto t_Y$ on $\Oc_X$ has at least two elements, proving \ref{orbit:lower}.

The action factors through $G:=\Gal(\QQ(\zeta_N)/\QQ)$, $N=\FSexp(\C)$, and $w(X)\in\QQ(\zeta_N)$. The orbit map $G\to\Oc_X$ has fibers of size $|\Stab_G(X)|$, so $\sum_{Y\in\Oc_X}w(Y)=|\Stab_G(X)|^{-1}\sum_{\sigma\in G}\sigma(w(X))=|\Oc_X|\,\Av(w(X))$. Let $K:=\QQ(\dim(\C),\dim(X)^2)$. Since $\dim(\C)=w(X)^{-1}\dim(X)^2$ is a product of algebraic integers in $K$, we have $N_{K/\QQ}(w(X)^{-1})\leq N_{K/\QQ}(\dim(\C))$, i.e.\ $N_{\QQ(w(X))/\QQ}(w(X)^{-1})^{1/e}\leq\Ndim(\C)^{1/d}$, where $e:=[\QQ(w(X)):\QQ]$. The inequality of arithmetic and geometric means then gives $\Av(w(X))\geq\Ndim(\C)^{-1/d}$, and $\sum_{Y\in\Oc_X}w(Y)\leq1$ proves \ref{orbit:upper}.

If $\dim(\C)\in\ZZ$, then $\Av(w(X))=\Av(\dim(X)^2)/\dim(\C)$, which gives the first claim of \ref{orbit:integer}. If $\Oc_X=\Oc(\C)$, then $\one\in\Oc_X$, so every simple object is of the form $\hs(\one)$ and has squared dimension $\dim(\C)/\sigma(\dim(\C))=1$ by \eqref{eq:dimgalois}. All dimensions are then rational, so $\Oc(\C)=\Oc_\one=\{\one\}$ by \ref{G:unit}, i.e.\ $\C\simeq\mathrm{Vec}$.

If all dimensions are integers, then so is $\dim(\C)$, and \eqref{eq:dimgalois} gives $\dim(\hs(X))^2=\sigma(\dim(X)^2)=\dim(X)^2$. Moreover $\hs(\one)=\one$ by \ref{G:unit}, so $\sigma^2(t_\one)=t_\one$ by \ref{G:t}, and $\theta_{\hs(X)}=t_{\hs(X)}/t_\one=\sigma^2(\theta_X)$. The argument for \ref{orbit:lower}, with $\theta$ in place of $t$, proves \ref{orbit:intdims}.
\end{proof}

Part \ref{orbit:upper} refines \cite[Proposition 3.2]{Sch-norm}, which gives $d\leq\rank(\C)\leq\Ndim(\C)^{1/d}$; lower bounds for the rank of the type \ref{orbit:lower}, and the following observation, also appear in \cite[Remark 2.1]{Yu-p2}. If $q$ is a prime power dividing $\FSexp(\C)=\ord(T)$, then $q$ divides $\ord(\theta_Y)$ for some $Y\in\Oc(\C)$, and as $\theta_Y=t_Y/t_\one$,
\begin{equation}\label{eq:existsX}
q\text{ divides }\ord(t_X)\text{ for some }X\in\{Y,\one\}.
\end{equation}

\subsection{Examples}\label{sec:examples}

\begin{example}\label{ex:vec}
Let $G$ be a finite group and $\omega\in H^3(G,\CC^\times)$. The categories $\VecG{G}{\omega}$ of $G$-graded finite-dimensional complex vector spaces are pseudounitary. For $g\in G$, let $|g|$ be the order of $g$ and $|\omega_g|$ the order of the restriction of $\omega$ to $\langle g\rangle$. The Frobenius--Schur exponent of $\VecG{G}{\omega}$ was computed in \cite[Proposition 7.1]{NS2}: $\FSexp(\VecG{G}{\omega})=\FSexp(\Zc(\VecG{G}{\omega}))=\lcm\{|g||\omega_g|:g\in G\}$. As $H^3(C_n,\CC^\times)\cong C_n$, $|\omega_g|$ divides $|g|$, so $\FSexp(\VecG{G}{\omega})$ divides $\exp(G)^2$. Hence $\FSexp(\VecG{G}{\omega})\leq|G|^2=\Ndim(\VecG{G}{\omega})^2$, with equality if and only if $G$ is cyclic and $\omega$ generates $H^3(G,\CC^\times)$. As $\FSexp$ and $\dim$ depend only on the Drinfeld center, the same bound holds for group-theoretical fusion categories \cite[Section 8.8]{ENO}, which are categorically Morita equivalent to some $\VecG{G}{\omega}$ (cf.\ \cite[Theorem 9.2]{NS1}).
\end{example}

\begin{example}\label{ex:ising}
An \emph{Ising} fusion category is a fusion category $\C$ with $\FPdim(\C)=2^2$ which is not pointed. Its simple objects are $\one,\delta,X$, with $\delta\otimes\delta\cong\one$, $\delta\otimes X\cong X$, and $X\otimes X\cong\one\oplus\delta$. There are precisely $2$ Ising fusion categories \cite[Proposition B.5]{DGNO}, each with $2$ spherical structures ($\dim(X)=\pm\sqrt2$) and $4$ braidings, all of them nondegenerate, with $\theta_\delta=-1$ and $\theta_X$ a primitive $16$th root of unity \cite[Appendix B]{DGNO}. Hence every Ising modular fusion category $\I$ has $\FSexp(\I)=2^4=\Ndim(\I)^2$.
\end{example}

\begin{example}\label{ex:qg}
For a simple Lie algebra $\mathfrak{g}$ and a root of unity $q$ such that $q^2$ is a primitive $\ell$th root of unity, let $\C(\mathfrak{g},\ell,q)$ be the associated premodular fusion category \cite[Section 4]{Sch-primer} (for $\mathfrak{g}=\sltwo$, the category at level $\ell-2$), and $\C(\mathfrak{g},\ell,q)_\ad$ its adjoint subcategory, which is modular in the cases below. We call a modular fusion category of rank $2$ with fusion rule $X\otimes X\cong X\oplus\one$ a \emph{Fibonacci} modular fusion category. There are two fusion categories with this fusion rule; they are Galois conjugate, each admits a unique spherical structure, and together with the two pointed ones they are all fusion categories of rank $2$ \cite{Ostrik-rank2}. For $\mathfrak{g}=\sltwo$ and $\ell=5$, $\C(\sltwo,5,q)_\ad$ is a Fibonacci modular fusion category with $\dim(X)=(1\pm\sqrt5)/2$, global dimension $(5\pm\sqrt5)/2$, and $\theta_X=q^4$, so $\FSexp(\C(\sltwo,5,q)_\ad)=5=\Ndim(\C(\sltwo,5,q)_\ad)$.

For $\mathfrak{g}=\sofive$ and $\ell=9$, the fusion rules of $\C(\sofive,9,q)_\ad$ cannot be categorified by any pseudounitary fusion category \cite[Example 1]{Sch-nonpu}. With $\zeta:=q^2$ and $u_k:=\zeta^k+\zeta^{-k}$, its simple objects have dimensions $1,-1,1,u_1,u_2,u_4$ and twists $1,\zeta^6,\zeta^3,\zeta^5,\zeta^8,\zeta^2$, respectively, so $\FSexp(\C(\sofive,9,q)_\ad)=\dim(\C(\sofive,9,q)_\ad)=\Ndim(\C(\sofive,9,q)_\ad)=3^2$.
\end{example}

\section{Prime power Frobenius--Perron dimension}\label{sec:FPdim}

Fusion categories $\C$ with $\FPdim(\C)\in\{p,p^2\}$ for a prime $p$ are pointed \cite[Corollary 8.30, Proposition 8.32]{ENO}, unless $p=2$ and $\C$ is one of the $2$ Ising fusion categories (Example \ref{ex:ising}). More generally, every nontrivial fusion category whose Frobenius--Perron dimension is a prime power has a nontrivial grading \cite[Theorem 8.28]{ENO}; for modular fusion categories this means $\C_\pt$ is nontrivial. For odd $p$ and $n\leq4$, all modular fusion categories $\C$ with $\FPdim(\C)=p^n$ are pointed \cite{CP} (this also follows from Lemma \ref{lem:oddstructure} below), and we recall below that the same is true for $p=2$ if $\C$ is integral. We prove the following classification.

\begin{proposition}\label{prop:classification}
Let $p$ be a prime, $n\leq5$, and $\C$ a modular fusion category with $\FPdim(\C)=p^n$. Then one of the following holds.
\begin{enumerate}
\item\label{class:pointed} $\C$ is pointed.
\item\label{class:even} $p=2$ and $\C$ is braided equivalent to
\begin{enumerate}
\item\label{class:Q8} a modular fusion subcategory of Frobenius--Perron dimension $2^5$ of $\Zc(\VecG{Q_8}{\omega})$, where $\omega\in H^3(Q_8,\CC^\times)$ has order $2^3$ \cite[Section 3]{Sch-neargroup},
\item\label{class:ising1} a product of an Ising modular fusion category and a pointed modular fusion category of Frobenius--Perron dimension $1$, $2$, $2^2$, or $2^3$, or
\item\label{class:ising2} a product of two Ising modular fusion categories and a pointed modular fusion category of Frobenius--Perron dimension $1$ or $2$.
\end{enumerate}
\item\label{class:odd} $p$ is odd, $n=5$, $\FSexp(\C)\in\{p,p^2\}$, and $\C$ is a modular fusion subcategory of $\Zc(\VecG{G}{\omega})$, where $G$ is extraspecial of order $p^3$ and $\omega\in H^3(G,\CC^\times)$ has order $p$.
\end{enumerate}
\end{proposition}

Here $\C$ is weakly integral and carries its canonical spherical structure (Section \ref{sec:fusion}), so this classifies braided fusion categories. Proposition \ref{prop:classification} is proved in Section \ref{sec:FPeven} for $p=2$ and follows from Lemma \ref{lem:oddstructure} and Proposition \ref{prop:odd5} for odd $p$.

\subsection{The even prime}\label{sec:FPeven}

Assume first $\FPdim(\C)=2^3$ and $\C$ is not pointed. Noninvertible simple objects have squared dimension $2$ or $2^2$, as $\dim(X)^2\in\ZZ$ divides $\dim(\C)$ \cite[Proposition 8.14.6]{EGNO}. If some $X\in\Oc(\C)$ has $\dim(X)=2$, then $\C$ is integral (otherwise the integral component of the grading of \cite[Theorem 3.10]{GN} would contain $\one$ and $X$, so its dimension would exceed $\FPdim(\C)/2=2^2$), so $X$ is the unique noninvertible simple object, and $\Stab(X)=\Oc(\C_\pt)$ has order $4$, $X\otimes X^\ast$ is the sum of all invertible objects, and $\C_\ad=\C_\pt$. Hence $\C_\pt=C_\C(\C_\pt)$ is symmetrically braided of dimension $4$, so it contains a Tannakian subcategory of dimension at least $2$ \cite[Corollary 9.9.32]{EGNO}, whose nontrivial objects have trivial twist, contradicting Lemma \ref{lem:central}. So all noninvertible simple objects have dimension $\sqrt2$, and $\C_\pt$ is the integral component of the grading of \cite[Theorem 3.10]{GN}, of dimension $2^2$. Thus $\C$ has two simple objects of dimension $\sqrt2$, and $|\Oc(\C_\ad)|=2$. All such braided fusion categories are classified \cite[Theorem 5.5]{DNS}; in particular $\C\simeq\I\boxtimes\D$ with $\I$ an Ising and $\D$ a pointed modular fusion category of rank $2$.

The modular fusion categories of Frobenius--Perron dimension $2^4$ are described in \cite[Lemma 4.10]{Yu-slightly}: any such category is pointed or braided equivalent to $\I\boxtimes\D$, where $\I$ is an Ising modular fusion category and $\D$ is pointed or Ising. Likewise, those of Frobenius--Perron dimension $2^5$ which are not integral are described in \cite[Corollary 4.13]{Yu-slightly}: any such category is braided equivalent to $\I_1\boxtimes\I_2\boxtimes\Pc_1$ or $\I_1\boxtimes\Pc_2$, where $\I_1,\I_2$ are Ising modular fusion categories and $\Pc_1,\Pc_2$ are pointed modular fusion categories of Frobenius--Perron dimension $2$ and $2^3$, respectively.

It remains to consider modular fusion categories $\C$ with $\FPdim(\C)=2^5$ which are integral but not pointed. Every $X\in\Oc(\C)$ has $\FPdim(X)^2$ dividing $2^5$, so $\FPdim(X)\in\{1,2,2^2\}$. Let $U$ be the universal grading group, so $|U|=|\Oc(\C_\pt)|$, and each graded component has dimension $2^5/|U|$. The grading is nontrivial and $\C$ is not pointed, so $|U|\in\{2,2^2,2^3,2^4\}$. If $|U|=2^4$, each component has dimension $2$, so all simple objects are invertible. If $|U|=2$, then $2^5=2+2^2b+2^4c$ for some $b,c\in\ZZ_{\geq0}$, which is impossible modulo $4$. Thus $|U|\in\{2^2,2^3\}$, every component has dimension $2^3$ or $2^2$, and every noninvertible simple object has dimension $2$.

If $|U|=2^2$, then $2^3=\FPdim(\C_\ad)=a+2^2b$, where $a$, the number of invertible objects of $\C_\ad$, divides $|\Oc(\C_\pt)|=2^2$, and $b$ is the number of noninvertible ones. Hence $a=2^2$ and $b=1$, so $\C_\ad\supset\C_\pt$ is a near-group fusion category whose M\"uger center $\C_\ad\cap C_\C(\C_\ad)=\C_\pt$ has dimension $2^2$. Thus $\C$ is a minimal modular extension of the nonsymmetrically braided near-group fusion category $\C_\ad$. There are exactly $8$ such categories up to equivalence, realized as braided fusion subcategories of the twisted doubles $\Zc(\VecG{Q_8}{\omega})$, where $\omega\in H^3(Q_8,\CC^\times)$ is any generator \cite[Section 3]{Sch-neargroup}. If $|U|=2^3$, then $\FPdim(\C_\ad)=2^2$, so $\C_\ad=C_\C(\C_\pt)\subset\C_\pt$ is pointed and symmetrically braided, and contains an invertible $g\not\cong\one$ with $\theta_g=1$ \cite[Corollary 9.9.32]{EGNO}. Every noninvertible simple object $Y$ is the unique simple object of its graded component, so $g\otimes Y\cong Y$, contradicting Lemma \ref{lem:central}.

\subsection{The odd primes}\label{sec:FPodd}

\begin{lemma}\label{lem:oddstructure}
Let $\C$ be a modular fusion category with $\FPdim(\C)=p^n$ for an odd prime $p$. If $\C$ is not pointed, then $(\C_\ad)_\pt$ contains a fusion subcategory $\Rep(C_p\times C_p)$, and $p^2\leq\dim(\C_\pt)\leq p^{n-3}$.
\end{lemma}

\begin{proof}
The category $\C$ is integral \cite[Theorem 3.10]{GN}, so every noninvertible simple object has dimension divisible by $p$. The symmetrically braided category $(\C_\ad)_\pt=\C_\pt\cap C_\C(\C_\pt)$ has odd dimension, hence is Tannakian \cite[Corollary 9.9.32(i)]{EGNO}, so all of its objects have trivial twist. For noninvertible $X\in\Oc(\C)$, $|\Stab(X)|\equiv\dim(X\otimes X^\ast)\equiv0\pmod p$, so $\Stab(X)$ is a nontrivial subgroup of the $p$-group $\Oc((\C_\ad)_\pt)$. If $\Oc((\C_\ad)_\pt)$ were cyclic, its unique subgroup of order $p$ would be contained in $\Stab(X)$ for all noninvertible $X$, contradicting Lemma \ref{lem:central}. Hence $\Oc((\C_\ad)_\pt)$ is not cyclic, which gives the first claim and $\dim(\C_\pt)\geq p^2$.

For the upper bound, recall that the components of the universal grading have dimension $p^n/\dim(\C_\pt)$. Since $\C$ is not pointed, $\dim(\C_\pt)\neq p^n$, and $\dim(\C_\pt)\neq p^{n-1}$ as no component of dimension $p$ contains a noninvertible object. If $\dim(\C_\pt)=p^{n-2}$, then $\dim(\C_\ad)=p^2\leq\dim((\C_\ad)_\pt)$, so $\C_\ad$ is pointed. A noninvertible $X\in\Oc(\C)$ is then the unique simple object of its component, which is stable under tensoring with $\C_\ad$, so $\Stab(X)=\Oc(\C_\ad)$ for all noninvertible $X$, again contradicting Lemma \ref{lem:central}.
\end{proof}

\begin{proposition}\label{prop:odd5}
Let $\C$ be a modular fusion category with $\FPdim(\C)=p^n$ for an odd prime $p$ and $n\leq5$. Then $\C$ is pointed, or $n=5$ and $\FSexp(\C)\in\{p,p^2\}$. In the latter case, $\C\boxtimes\Pc\simeq\Zc(\VecG{G}{\omega})$ for a pointed modular fusion category $\Pc$ of dimension $p$, an extraspecial $p$-group $G$ of order $p^3$, and $\omega\in H^3(G,\CC^\times)$ of order $p$.
\end{proposition}

\begin{proof}
If $\C$ is not pointed, Lemma \ref{lem:oddstructure} implies $n=5$ and $\dim(\C_\pt)=p^2$, so $\C_\pt=(\C_\ad)_\pt\simeq\Rep(C_p\times C_p)$ is Tannakian. Each component of the universal grading has dimension $p^3$, so every noninvertible simple object has dimension $p$, and $\C_\ad$ has exactly $(p^3-p^2)/p^2=p-1$ noninvertible simple objects. Since the orbits of $\Oc(\C_\pt)$ acting on them have $p$-power size, $\Stab(X)=\Oc(\C_\pt)$ and $X\otimes X^\ast\cong\bigoplus_{g\in\Oc(\C_\pt)}g$ for each of them. Thus $(\C_\ad)_\ad=\C_\pt$, and $\C_\ad$ is $C_p$-graded with each nontrivial component consisting of a single simple object of dimension $p$. This shows that $\C_\ad$ has the fusion rules of the character ring of an extraspecial $p$-group of order $p^3$.

Next we show $\FSexp(\C)\leq p^2$. Suppose $\FSexp(\C)=p^m$ with $m\geq3$, and let $X\in\Oc(\C)$ with $\ord(\theta_X)=p^m$, so $|\Oc_X|\geq\epsilon_X\,s(p^m)\geq\epsilon_Xp^2(p-1)/2$ by Lemma \ref{lem:orbit}\ref{orbit:intdims}. As $\C_\pt$ is Tannakian, $X$ is not invertible. The set $\Oc(\C_\pt)$ is Galois stable by Lemma \ref{lem:orbit}\ref{orbit:intdims}, and so are $\Oc(\C_\ad)=\Oc(C_\C(\C_\pt))$ and $\Oc(C_\C(\F))$ for $\F\subset\C_\pt$ by \ref{G:grading}. If $X\in\C_\ad$, then $\Oc_X\subset\Oc(\C_\ad)\setminus\Oc(\C_\pt)$, so $p-1\geq|\Oc_X|\geq p^2(p-1)/2$, which is absurd. Otherwise $X$ lies in a nontrivial component $\C_\chi$, $\chi$ of odd order, so $X^\ast\in\C_{\chi^{-1}}$ and $\epsilon_X=2$. The group $\Oc(\C_\pt)$ of order $p^2$ acts on the $p$ simple objects of $\C_\chi$, so $\Stab(X)$ contains the simple objects of a fusion subcategory $\F\subset\C_\pt$ of dimension $p$; as $\F$ is Tannakian, $X\in C_\C(\F)$ by the balancing equation. Then $\Oc_X\subset\Oc(C_\C(\F))\setminus\Oc(\C_\ad)$, a set of $(p^4-p^3)/p^2=p(p-1)$ simple objects since $\dim(C_\C(\F))=p^4$ and $\C_\ad\subset C_\C(\F)$. So $p(p-1)\geq|\Oc_X|\geq p^2(p-1)$, which is absurd. Hence $\FSexp(\C)\in\{p,p^2\}$, and the remaining claim is \cite[Proposition 6.0.2]{Sch-neargroup}.
\end{proof}

\section{Prime power Frobenius--Schur exponent}\label{sec:FS}

We begin with a bound which holds for all modular fusion categories. Recall the function $s$ and the constants $c_p$ from Section \ref{sec:galois}.

\begin{proposition}\label{prop:general}
Let $\C$ be a modular fusion category and $d:=[\QQ(\dim(\C)):\QQ]$.
\begin{enumerate}
\item\label{general:q} If $q>1$ is a power of a prime $p$ and $q$ divides $\FSexp(\C)$, then $q\leq c_p\cdot s(q)\leq c_p\cdot\Ndim(\C)^{1/d}$. If moreover $d=1$, then $s(q)<\Ndim(\C)$.
\item\label{general:all} If $\Ndim(\C)$ has $r$ prime divisors, then $\FSexp(\C)\leq\big(\prod_{p\mid\Ndim(\C)}c_p\big)\cdot\Ndim(\C)^{r/d}\leq8\cdot3^{r-1}\cdot\Ndim(\C)^{r/d}$.
\end{enumerate}
\end{proposition}

\begin{proof}
Let $X$ be as in \eqref{eq:existsX}. Then $s(q)\leq s(\ord(t_X))\leq|\Oc_X|\leq\Ndim(\C)^{1/d}$ by Lemma \ref{lem:orbit}\ref{orbit:lower},\ref{orbit:upper}, and $q\leq c_p\cdot s(q)$ by \eqref{eq:cp}. If $d=1$, then $\C\not\simeq\mathrm{Vec}$ and $\Av(\dim(X)^2)\geq1$ give $|\Oc_X|\leq\dim(\Oc_X)<\dim(\C)$ by Lemma \ref{lem:orbit}\ref{orbit:integer}. This proves \ref{general:q}. The primes dividing $\FSexp(\C)$ are those dividing $\Ndim(\C)$ \cite[Theorem 3.9]{BNRW}, so \ref{general:all} follows by applying \ref{general:q} to the largest power of each such prime dividing $\FSexp(\C)$ and taking the product.
\end{proof}

\begin{remark}\label{rem:general}
Proposition \ref{prop:general}\ref{general:q} is sharp up to the constant $c_p$. For a prime $p\geq5$, the categories $\C(\sltwo,2p,q)^0_A$ obtained by condensing the boson of $\C(\sltwo,2p,q)$ have global dimension $p\cdot u$ for a totally positive unit $u$ of degree $d=(p-1)/2$, and all of their twists are $p$th roots of unity, not all trivial \cite[Example 2.2]{Yu-p2}. Hence $\FSexp(\C)=p=\Ndim(\C)^{1/d}$, while $\Ndim(\C)=p^{(p-1)/2}$; for $p=5$ these are the products of two Fibonacci modular fusion categories. Yu also shows that if $\FSexp(\C)$ is a power of a prime $p>3$, $\Ndim(\C)=p^n$, and $d>n$, then $\C$ is a Galois conjugate of $\C(\sltwo,p,q)_\ad$, and describes the modular fusion categories with $\Ndim(\C)=p^2$ \cite[Lemma 3.1, Theorem 3.5]{Yu-p2}. In Proposition \ref{prop:general}\ref{general:all}, the exponent $r/d$ arises because distinct primes may be detected by distinct Galois orbits, in which case Lemma \ref{lem:orbit} bounds the sum of the numbers $s(q)$ rather than their product; we do not know whether $\FSexp(\C)$ is bounded by a constant multiple of $\Ndim(\C)^{1/d}$ in general.
\end{remark}

In Theorem \ref{thm:main} and its proof, $\C$ is a modular fusion category and $\FSexp(\C)=p^m$ is a power of a prime $p$, $m\in\ZZ_{\geq1}$. Then $\Ndim(\C)=p^n$ for some $n\in\ZZ_{\geq1}$ \cite[Theorem 3.9]{BNRW}, and $\dim(\C)\in\QQ(\zeta_{p^m})^+$, so $d:=[\QQ(\dim(\C)):\QQ]$ divides $[\QQ(\zeta_{p^m})^+:\QQ]$, which is $p^{m-1}(p-1)/2$ for odd $p$ and $2^{m-2}$ for $p=2$, $m\geq2$.

\begin{theorem}\label{thm:main}
Let $\C$ be a modular fusion category with $\FSexp(\C)=p^m$ and $\Ndim(\C)=p^n$ as above.
\begin{enumerate}
\item\label{main:odd} If $p$ is odd, then $\FSexp(\C)\leq\Ndim(\C)$. If moreover $\FSexp(\C)=\Ndim(\C)$ and $\dim(\C)\notin\ZZ$, then $\C$ is a Fibonacci modular fusion category.
\item\label{main:even} If $p=2$, then $\FSexp(\C)\leq4\cdot\Ndim(\C)$. If moreover $\FSexp(\C)\geq\Ndim(\C)$, then $\dim(\C)\in\ZZ$.
\end{enumerate}
The bounds in \ref{main:odd} and \ref{main:even} are optimal, attained by the pointed categories $\C(C_{p^n},q)$ for odd $p$ and by Ising modular fusion categories for $p=2$ (Example \ref{ex:ising}).
\end{theorem}

\begin{proof}
Let $X$ be as in \eqref{eq:existsX} for $q=p^m$, so $|\Oc_X|\geq s(p^m)$ by Lemma \ref{lem:orbit}\ref{orbit:lower}.

If $d=1$, then $s(p^m)<p^n$ by Proposition \ref{prop:general}\ref{general:q}. For odd $p$ this reads $p^{m-1}(p-1)/2<p^n$, which forces $m\leq n$; for $p=2$ it forces $m\leq n+2$ by \eqref{eq:s2}.

Now assume $d\geq2$. Lemma \ref{lem:orbit}\ref{orbit:upper} and \ref{G:unit} give
\begin{equation}\label{eq:dge2}
s(p^m)\leq|\Oc_X|\leq p^{n/d}\qquad\text{and}\qquad d\leq|\Oc_\one|\leq p^{n/d}.
\end{equation}

\emph{Case $p\geq5$.} Suppose $m\geq n$. Then $2p^{n-1}\leq p^{n-1}(p-1)/2\leq s(p^m)\leq p^{n/2}$, so $n=1$; then $p^{m-1}(p-1)/2\leq p^{1/2}$ forces $m=1$, and $(p-1)/2\leq p^{1/d}$ forces $p=5$ and $d=2$. Hence $m\leq n$, and if $m=n$, then $p=5$, $n=1$, and $d=2$. In that case let $D:=\dim(\C)$ and $D'$ be its Galois conjugate. Then $DD'=5$, $D,D'\geq1$, and $D\notin\QQ$, so $D+D'\in\ZZ$ lies in $[2\sqrt5,6)$, i.e.\ $D+D'=5$ and $\{D,D'\}=\{(5\pm\sqrt5)/2\}$. By Lemma \ref{lem:orbit}\ref{orbit:upper}, $\sum_{Y\in\Oc_\one}w(Y)=|\Oc_\one|\cdot(1/D+1/D')/2=|\Oc_\one|/2\geq1$, so $\Oc(\C)=\Oc_\one$ has exactly two elements. Thus $\C$ has rank $2$ and is not pointed, so $\C$ is a Fibonacci modular fusion category (Example \ref{ex:qg}).

\emph{Case $p=3$.} Here $d\geq2$ is a power of $3$ dividing $3^{m-1}$, so $d\geq3$ and $m\geq2$. Then \eqref{eq:dge2} gives $3^{m-1}\leq3^{n/3}$, i.e.\ $m\leq n/3+1$. If $n=1$ this contradicts $m\geq2$, and if $n\geq2$ then $m<n$.

\emph{Case $p=2$.} Here $d\geq2$ is a power of $2$ dividing $2^{m-2}$, so $m\geq3$, and \eqref{eq:dge2} gives $m\leq n/d+3$ and $d^d\leq2^n$. If $d\geq4$, then $n\geq8$ and $m\leq n/4+3<n$. If $d=2$, then $n\geq2$ and $m\leq n/2+3\leq n+2$. It remains to show that $d=2$ and $m\geq n$ cannot occur. In that case $n\leq6$ since $n\leq m\leq n/2+3$, and $\QQ(\dim(\C))=\QQ(\sqrt2)$. Let $D:=\dim(\C)$ and $D'$ be its Galois conjugate, so $DD'=2^n$ and $D,D'\geq1$. As $2$ is totally ramified in $\ZZ[\sqrt2]$, $D=2^{n/2}\eta^k$ and $D'=2^{n/2}\eta^{-k}$ with $\eta:=1+\sqrt2$ and $k\equiv n\pmod2$; $k\neq0$ since $D\notin\QQ$, and $k\geq1$ after possibly exchanging $D$ and $D'$. The conditions $2^{n/2}\eta^{-k}\geq1$ and $n\leq6$ leave $(n,k)\in\{(3,1),(5,1),(6,2)\}$, for which $D+D'=8$, $16$, $48$, respectively. By Lemma \ref{lem:orbit}\ref{orbit:upper}, $\sum_{Y\in\Oc_\one}w(Y)=|\Oc_\one|\cdot(D+D')/2^{n+1}$ and, if $X\notin\Oc_\one$, $\sum_{Y\in\Oc_X}w(Y)\geq|\Oc_X|/2^{n/2}\geq2^{n/2-3}$. For $(n,k)=(3,1)$, $|\Oc_\one|\geq2$ gives $\sum_{Y\in\Oc_\one}w(Y)\geq1$. For $(n,k)=(5,1)$: if $X\notin\Oc_\one$, the two sums add up to at least $1/2+2^{-1/2}>1$; if $X\in\Oc_\one$, then $|\Oc_\one|\geq s(2^m)\geq4$ and $\sum_{Y\in\Oc_\one}w(Y)\geq1$. For $(n,k)=(6,2)$: if $X\notin\Oc_\one$, they add up to at least $3/4+1$; if $X\in\Oc_\one$, then $|\Oc_\one|\geq8$ and $\sum_{Y\in\Oc_\one}w(Y)\geq3$. In every case either $\sum_{Y\in\Oc(\C)}w(Y)>1$, which is absurd, or $\Oc(\C)=\Oc_\one$, i.e.\ $\C$ is a nontrivial transitive modular fusion category. The latter is impossible, since the Frobenius--Schur exponent of a transitive modular fusion category is a product of primes greater than $3$ \cite[Theorem I and Theorem 3.11]{NWZ}.
\end{proof}

\subsection{Pseudounitary fusion categories}\label{sec:FSpu}

\begin{lemma}\label{lem:oddpointed}
Let $\C$ be a modular fusion category with $\FPdim(\C)=p^n$ for an odd prime $p$ and $n\in\ZZ_{\geq1}$. If $\FSexp(\C)=\Ndim(\C)$, then $\C$ is pointed.
\end{lemma}

\begin{proof}
The category $\C$ is integral and pseudounitary \cite[Theorem 3.10]{GN}, \cite[Propositions 8.23, 8.24]{ENO}, so $\Ndim(\C)=\dim(\C)=p^n=\FSexp(\C)$, and all dimensions are odd integers. Let $X\in\Oc(\C)$ with $\ord(\theta_X)=p^n$. By Lemma \ref{lem:orbit}\ref{orbit:integer},\ref{orbit:intdims}, $p^n>\dim(\Oc_X)=|\Oc_X|\dim(X)^2\geq\epsilon_Xp^{n-1}(p-1)\dim(X)^2/2$. As $\dim(X)^2\in\{1\}\cup[p^2,\infty)$, $X$ is invertible; as $\Oc(\C_\pt)$ has odd order, $X\not\cong X^\ast$, so $\epsilon_X=2$; and $\Oc_X\subset\Oc(\C_\pt)$ by Lemma \ref{lem:orbit}\ref{orbit:intdims}. Hence $\dim(\C_\pt)>p^{n-1}(p-1)\geq p^{n-1}$. As $\dim(\C)/\dim(\C_\pt)=\dim(\C_\ad)$ is a rational algebraic integer (Section \ref{sec:fusion}), $\dim(\C_\pt)$ divides $\dim(\C)$, so $\C=\C_\pt$.
\end{proof}

\begin{lemma}\label{lem:evenlist}
Let $\C$ be a pseudounitary modular fusion category with $\FPdim(\C)=2^n$ for some $n\in\ZZ_{\geq1}$. Then $\FSexp(\C)<\Ndim(\C)$, or one of the following holds.
\begin{enumerate}
\item\label{even:one} $\FSexp(\C)=\Ndim(\C)$ and $\C$ is pointed, $\C\simeq\I\boxtimes\D$ where $\I$ is an Ising modular fusion category and $\D$ is pointed of rank $2^2$, or $\C\simeq\I\boxtimes\I'$ where $\I,\I'$ are Ising modular fusion categories.
\item\label{even:two} $\FSexp(\C)=2\cdot\Ndim(\C)$ and $\C$ is pointed, or $\C\simeq\I\boxtimes\D$ where $\I$ is an Ising modular fusion category and $\D$ is pointed of rank $2$.
\item\label{even:four} $\FSexp(\C)=2^2\cdot\Ndim(\C)$ and $\C$ is an Ising modular fusion category.
\end{enumerate}
\end{lemma}

\begin{proof}
Pointed modular fusion categories $\D$ of dimension $2^n$ have $\FSexp(\D)\leq2^{n+1}$, as $\theta_g^{2\ord(g)}=1$ for invertible $g$ (Section \ref{sec:fusion}). For an Ising modular fusion category $\I$, $\FSexp(\I\boxtimes\D)=\lcm(2^4,\FSexp(\D))$. Hence if $n\leq5$, the lemma follows from Proposition \ref{prop:classification}: the categories in Proposition \ref{prop:classification}\ref{class:Q8} have Frobenius--Schur exponent dividing $\FSexp(\Zc(\VecG{Q_8}{\omega}))\leq2^4$ by Example \ref{ex:vec}, and the products in \ref{class:ising1} and \ref{class:ising2} of dimension $2^5$ have Frobenius--Schur exponent $2^4$.

Now let $n\geq6$, assume the lemma holds for all smaller exponents, and suppose $\FSexp(\C)=2^m\geq2^n$; we show that $\C$ is pointed. By Theorem \ref{thm:main}\ref{main:even}, $n\leq m\leq n+2$. Simple objects have dimension $2^b$ or $2^b\sqrt2$, so $|\Oc_\one|\leq2$ by \ref{G:unit}, and Lemma \ref{lem:orbit}\ref{orbit:lower} for $X=\one$ gives $s(\ord(t_\one))\leq2$, i.e.\ $2^5\nmid\ord(t_\one)$ by \eqref{eq:s2}.

\emph{Step 0: $\C$ has no self-dual simple object of dimension $\sqrt2$.} Otherwise such an object generates an Ising braided fusion subcategory $\I$, which is modular \cite[Appendix B]{DGNO}, so $\C\simeq\I\boxtimes\C'$ with $\C'$ pseudounitary modular of dimension $2^{n-2}$, and $2^m=\lcm(2^4,\FSexp(\C'))$. As $m\geq6$, $\FSexp(\C')=2^m\geq2^2\cdot\dim(\C')$, so $\C'$ is an Ising modular fusion category by the lemma for $\C'$, i.e.\ $n=4$, a contradiction.

\emph{Step 1.} Let $Y\in\Oc(\C)$ with $\ord(\theta_Y)=2^m$. Since $\theta_Y=t_Y/t_\one$ and $2^5\nmid\ord(t_\one)$, $2^m$ divides $\ord(t_Y)$. Write $\dim(Y)^2=2^a$. Lemma \ref{lem:orbit} gives $\epsilon_Y2^{m-3+a}\leq\dim(\Oc_Y)<2^n$, i.e.\ $a+(m-n)+\log_2\epsilon_Y\leq2$. So $a\leq2$; if $a=1$ then $\epsilon_Y=2$ by Step 0, and $m=n$; if $a=2$ then $\epsilon_Y=1$ and $m=n$.

\emph{Step 2: $a=0$.} Then $Y$ is invertible and $\ord(Y)\geq2^{m-1}$. If $m\geq n+1$, then $\C=\C_\pt$ is cyclic. If $m=n$ and $\C$ is not pointed, then $\Oc(\C_\pt)$ is cyclic of order $2^{n-1}$, generated by $Y$, with $\ord(\theta_Y)=2^n$. The components of the universal grading have dimension $2$, so $\C_\ad=\{\one,\delta\}$ with $\delta=Y^{2^{n-2}}$ the unique invertible object of order $2$. But $\delta\in\C_\ad=C_\C(\C_\pt)$, while $b(Y,\delta)=\theta_Y^{2^{n-1}}=-1$, a contradiction.

\emph{Step 3: $a=2$.} Then $Y\cong Y^\ast$, so $\chi_Y=\chi_Y^{-1}$ takes values $\pm1$, and by \ref{G:grading} $\Oc_Y$ lies in the single component $\C_{\chi_Y}$. As $\dim(\Oc_Y)\geq2^{n-3}\cdot2^2$, this component has dimension at least $2^{n-1}$; as all components of the universal grading have the same dimension and the grading is nontrivial \cite[Theorem 8.28]{ENO}, $\dim(\C_\pt)=2$. Now $\C_\ad$ lies in the integral component of the grading of \cite[Theorem 3.10]{GN}, so $\C_\ad$ is integral, $\FPdim(\C_\ad)=2^{n-1}\equiv0\pmod4$, and all its noninvertible simple objects have even dimension. Hence $4$ divides $|\Oc((\C_\ad)_\pt)|\leq|\Oc(\C_\pt)|=2$, a contradiction.

\emph{Step 4: $a=1$.} Then $Y\not\cong Y^\ast$, $m=n$, and $|\Oc_Y|\geq2^{n-2}$. Let $\C=\C_0\oplus\C_1$ be the grading of \cite[Theorem 3.10]{GN}, with $\C_0$ integral and $Y\in\C_1$, so $\dim(\C_0)=\dim(\C_1)=2^{n-1}$. As $\Oc_Y\subset\Oc(\C_1)$ and $\dim(\Oc_Y)\geq2^{n-1}$, $\Oc(\C_1)=\Oc_Y$. Let $2^c:=\dim(\C_\pt)$. Each component of the universal grading lies in $\C_0$ or $\C_1$, so $\C_1$ is a union of $2^{c-1}$ components, which by \ref{G:grading} are indexed by the conjugates $\sigma\circ\chi_Y$. If $\chi_Y$ has order $2^j$, there are $2^{j-1}$ of these, so $j=c$ and $\Oc(\C_\pt)=\langle h\rangle$ is cyclic.

Now $\dim(Y\otimes Y)=2$. The braiding $c_{Y,Y}$ acts on $\Hom(Z,Y\otimes Y)$ by scalars $\lambda$ with $\lambda^2=\theta_Z\theta_Y^{-2}$, and $\Tr(c_{Y,Y})=\theta_Y\dim(Y)=\sqrt2\,\theta_Y$. If $Y\otimes Y$ were simple, then $|\Tr(c_{Y,Y})|=2$; hence $Y\otimes Y\cong g\oplus g'$ for invertible $g,g'$ with $\lambda_g+\lambda_{g'}=\sqrt2\,\theta_Y$ and $|\lambda_g|=|\lambda_{g'}|=1$. So $\lambda_g=\theta_Ye^{\pm\pi i/4}$ and $\theta_g=\lambda_g^2\theta_Y^2=\pm i\,\theta_Y^4$ has order $2^{n-2}$.

Let $2^e:=\ord(\theta_h)$ and $R:=\Oc(\C_\pt)\cap\Oc(\C_\ad)$, the radical of $b$, since $\C_\ad=C_\C(\C_\pt)$. As $b(h,h^k)=\theta_h^{2k}$, $|R|=2^{c-e+1}$. Writing $g=h^\ell$, $\theta_g=\theta_h^{\ell^2}$ has order $2^{e-2v}$ with $v$ the $2$-adic valuation of $\ell$, so $e\equiv n\pmod 2$ and $e\geq n-2$. The nontrivial summand of $Y\otimes Y^\ast$ lies in $R$, so $|R|\geq2$, i.e.\ $e\leq c$; also $c\leq n-1$ since $\C$ is not pointed. If $c=n-1$, then $e=n-2$ and $|R|=4>2=\dim(\C_\ad)$. If $c=n-2$, then $e=n-2$ and $|R|=2$, but $\C_\ad\subset\C_0$ is integral of dimension $4$, hence pointed, so $|R|=4$. Both cases are contradictions, which completes the induction.

Finally, a pointed $\C$ with $\FSexp(\C)\geq2^n$ satisfies \ref{even:one} or \ref{even:two}, as $\FSexp(\C)\leq2^{n+1}$.
\end{proof}

\begin{corollary}\label{cor:modular}
Let $\C$ be Galois conjugate to a pseudounitary modular fusion category with $\FSexp(\C)$ a power of a prime $p$. Then $\FSexp(\C)<\Ndim(\C)$ unless
\begin{enumerate}
\item $\C$ is pointed, and if $p$ is odd then $\C\simeq\C(C_N,q)$, where $N=\FSexp(\C)$,
\item $\C$ is a Fibonacci modular fusion category (Example \ref{ex:qg}),
\item $\C\simeq\I\boxtimes\D$ for an Ising modular fusion category $\I$ and a pointed modular fusion category $\D$ of dimension $1$, $2$, or $2^2$, or
\item $\C\simeq\I\boxtimes\I'$ for Ising modular fusion categories $\I,\I'$.
\end{enumerate}
\end{corollary}

\begin{proof}
Galois conjugation preserves $\FSexp$, $\Ndim$, and the list above, so we may assume $\C$ is pseudounitary. Assume $\FSexp(\C)\geq\Ndim(\C)$. By Theorem \ref{thm:main}, either $\C$ is a Fibonacci modular fusion category, or $\FPdim(\C)=\dim(\C)=\Ndim(\C)=p^n$. In the latter case, for odd $p$, Lemma \ref{lem:oddpointed} shows $\C\simeq\C(A,q)$ is pointed with $|A|=\FSexp(\C)$. Each $g\in A$ has odd order $\ell$, and $\theta_g^{2\ell}=1=\theta_{g^\ell}=\theta_g^{\ell^2}$ by Section \ref{sec:fusion}, so $\theta_g^\ell=1$. Hence $\FSexp(\C)$ divides $\exp(A)$, and $A$ is cyclic. For $p=2$, the claim is Lemma \ref{lem:evenlist}.
\end{proof}

\begin{corollary}\label{cor:spherical}
Let $\C$ be Galois conjugate to a pseudounitary fusion category such that $\FSexp(\C)$ is a power of a prime $p$. Then $\FSexp(\C)\leq\Ndim(\C)^2$, with equality if and only if $\C\simeq\VecG{G}{\omega}$ for a cyclic $p$-group $G$ and a generator $\omega\in H^3(G,\CC^\times)$, or $p=2$ and $\C$ is an Ising fusion category.
\end{corollary}

\begin{proof}
The Drinfeld center $\Zc(\C)$ is Galois conjugate to a pseudounitary modular fusion category, with $\FSexp(\Zc(\C))=\FSexp(\C)$ and $\Ndim(\Zc(\C))=\Ndim(\C)^2$ (Section \ref{sec:fusion}). Every simple object of $\C$ is a summand of the image of a simple object of $\Zc(\C)$ under the forgetful functor \cite[Section 7.13]{EGNO}, so $\C$ is pointed (resp.\ integral) if $\Zc(\C)$ is, and $\Zc(\C)$ is integral if $\C$ is.

Assume $\FSexp(\C)\geq\Ndim(\C)^2$. Then $\FSexp(\Zc(\C))\geq\Ndim(\Zc(\C))$, so $\Zc(\C)$ is one of the categories in Corollary \ref{cor:modular}, and $\Ndim(\Zc(\C))$ is a perfect square. This excludes the Fibonacci categories and $\I\boxtimes\D$ with $\dim(\D)=2$. If $\Zc(\C)$ is pointed, then $\C\simeq\VecG{G}{\omega}$ is pointed, and Example \ref{ex:vec} gives $\FSexp(\C)\leq\Ndim(\C)^2$, with equality exactly in the stated cases. Otherwise $\Zc(\C)$ is weakly integral but not integral, with $\FPdim(\Zc(\C))\in\{2^2,2^4\}$, so $\C$ is not integral and $\FPdim(\C)\in\{2,2^2\}$. Hence $\C$ is an Ising fusion category (Section \ref{sec:FPdim}), and $\FSexp(\C)=2^4=\Ndim(\C)^2$ by Example \ref{ex:ising}. Conversely, the categories in the statement satisfy $\FSexp(\C)=\Ndim(\C)^2$ by Examples \ref{ex:vec} and \ref{ex:ising}.
\end{proof}

\subsection{Arbitrary spherical fusion categories}\label{sec:FSsph}

Without pseudounitarity, the bounds of Proposition \ref{prop:general} and Theorem \ref{thm:main} still pass to Drinfeld centers.

\begin{corollary}\label{cor:general}
Let $\C$ be a spherical fusion category and $d:=[\QQ(\dim(\C)):\QQ]$.
\begin{enumerate}
\item\label{sph:q} If $q>1$ is a power of a prime $p$ and $q$ divides $\FSexp(\C)$, then $q\leq c_p\cdot s(q)\leq c_p\cdot\Ndim(\C)^{2/d}$. Consequently $\FSexp(\C)\leq8\cdot3^{r-1}\cdot\Ndim(\C)^{2r/d}$, where $r$ is the number of primes dividing $\Ndim(\C)$.
\item\label{sph:pp} If $\FSexp(\C)$ is a power of a prime $p$, then $\FSexp(\C)\leq\Ndim(\C)^2$ if $p$ is odd, and $\FSexp(\C)\leq4\cdot\Ndim(\C)^2$ if $p=2$. If moreover $\FSexp(\C)\geq\Ndim(\C)^2$, then $\dim(\C)\in\ZZ$.
\end{enumerate}
\end{corollary}

\begin{proof}
By Section \ref{sec:fusion}, \ref{sph:q} is Proposition \ref{prop:general} for $\Zc(\C)$, and the bounds in \ref{sph:pp} are those of Theorem \ref{thm:main}. If $\FSexp(\C)\geq\Ndim(\C)^2$, then $\Zc(\C)$ is not a Fibonacci modular fusion category, as $\Ndim(\Zc(\C))$ is a perfect square, so $\dim(\C)^2\in\ZZ$ by Theorem \ref{thm:main}; thus $\QQ(\dim(\C))=\QQ$ and $\dim(\C)\in\ZZ$.
\end{proof}

By Corollary \ref{cor:spherical}, the bound $4\cdot\Ndim(\C)^2$ for $p=2$ is not attained by Galois conjugates of pseudounitary categories; we do not know if it is attained at all.

\subsection{Further consequences}\label{sec:FSfurther}

\subsubsection*{Twist orders and formal codegrees}
Lemma \ref{lem:orbit} also gives information about individual simple objects. For a simple object $X$ of a modular fusion category, $f_X:=\dim(\C)/\dim(X)^2$ is the formal codegree of $\C$ associated with $X$.

\begin{corollary}\label{cor:codegree}
Let $\C$ be a modular fusion category with $\dim(\C)\in\ZZ$, and let $X\in\Oc(\C)$, $X\not\cong\one$.
\begin{enumerate}
\item\label{codeg:t} For every prime power $q$ dividing $\ord(t_X)$, $\epsilon_X\cdot s(q)\cdot\Av(\dim(X)^2)<\dim(\C)$.
\item\label{codeg:theta} If $\dim(Y)\in\ZZ$ for all $Y\in\Oc(\C)$, e.g.\ if $\C$ is integral, then $\epsilon_X\cdot s(\ord(\theta_X))<f_X$. In particular every prime power $q$ dividing $\ord(\theta_X)$ satisfies $q<c_p\cdot f_X/\epsilon_X$.
\end{enumerate}
\end{corollary}

\begin{proof}
As $X\not\cong\one$, $\C\not\simeq\mathrm{Vec}$, so Lemma \ref{lem:orbit} gives $\epsilon_X\,s(q)\,\Av(\dim(X)^2)\leq|\Oc_X|\,\Av(\dim(X)^2)=\dim(\Oc_X)<\dim(\C)$, which proves \ref{codeg:t}. If all dimensions are integers, then Lemma \ref{lem:orbit}\ref{orbit:integer},\ref{orbit:intdims} give $\epsilon_X\,s(\ord(\theta_X))\dim(X)^2\leq|\Oc_X|\dim(X)^2=\dim(\Oc_X)<\dim(\C)$, and the last claim follows from \eqref{eq:cp}.
\end{proof}

Lemma \ref{lem:oddpointed} and the case $d=1$ of Theorem \ref{thm:main} are instances of Corollary \ref{cor:codegree}. Part \ref{codeg:theta} is sharp: in a pointed modular fusion category $\C(C_p,q)$, $p$ an odd prime, every $X\not\cong\one$ has $f_X=p$, $\epsilon_X=2$, and $\ord(\theta_X)=p$, so \ref{codeg:theta} reads $p-1<p$.

\subsubsection*{Premodular fusion categories}
A premodular fusion category $\C$ is \emph{modularizable} if every transparent simple object $X$ satisfies $\theta_X=1$ and $\dim(X)>0$. Then the M\"uger center $\E$ of $\C$ is a Tannakian category $\Rep(G)$, and the de-equivariantization $\C_G$ is a modular fusion category with $\dim(\C_G)=\dim(\C)/\dim(\E)$, together with a dominant ribbon functor $\C\to\C_G$ \cite[Section 4.4]{DGNO}.

\begin{corollary}\label{cor:premodular}
Let $\C$ be a modularizable premodular fusion category with M\"uger center $\E$, and $d:=[\QQ(\dim(\C)):\QQ]$. If a power $q>1$ of a prime $p$ divides $\ord(T_\C)$, then $q\leq c_p\cdot\Ndim(\C)^{1/d}/\dim(\E)$. If $\ord(T_\C)$ is a power of $p$, then $\ord(T_\C)\leq\Ndim(\C)/\dim(\E)^d$ for odd $p$, and $\ord(T_\C)\leq4\cdot\Ndim(\C)/\dim(\E)^d$ for $p=2$.
\end{corollary}

\begin{proof}
Every simple object of $\C_G$ is a summand of the image of a simple object $X$ of $\C$, and has twist $\theta_X$; and the image of every $X$ is nonzero. Hence $\ord(T_\C)=\ord(T_{\C_G})=\FSexp(\C_G)$. As $\dim(\C_G)$ is a rational multiple of $\dim(\C)$, $\QQ(\dim(\C_G))=\QQ(\dim(\C))$ and $\Ndim(\C_G)=\Ndim(\C)/\dim(\E)^d$. Now apply Proposition \ref{prop:general} and Theorem \ref{thm:main} to $\C_G$.
\end{proof}

\subsubsection*{Nilpotent modular fusion categories}
Kashina's conjecture states that the exponent of a semisimple Hopf algebra $H$ divides $\dim(H)$ (see \cite{EG-exp}). Since $\exp(H)=\FSexp(\Rep(H))$ \cite{NS1}, the natural categorical analogue asks whether $\FSexp(\C)$ divides $\FPdim(\C)$ for integral fusion categories. The semion category shows that the analogue fails without a fiber functor: it is pointed of dimension $2$ with Frobenius--Schur exponent $4$. Our results give the following divisibility for nilpotent modular fusion categories, which need not admit fiber functors and need not be integral.

\begin{corollary}\label{cor:nilpotent}
Let $\C$ be a nilpotent modular fusion category, equipped with its canonical spherical structure, and let $\C\simeq\boxtimes_p\C_p$ be its decomposition into modular fusion categories with $\FPdim(\C_p)$ a power of $p$ \cite[Theorem 1.1]{DGNO-nilp}. Then $\FSexp(\C)$ divides $4\cdot\dim(\C)$. Moreover:
\begin{enumerate}
\item\label{nilp:div} $\FSexp(\C)$ divides $\dim(\C)$ unless $\C_2$ is pointed with $\FSexp(\C_2)=2\dim(\C_2)$, $\C_2\simeq\I\boxtimes\D$ with $\I$ an Ising and $\D$ a pointed modular fusion category of rank $2$, or $\C_2$ is an Ising modular fusion category;
\item\label{nilp:eq} $\FSexp(\C)=4\cdot\dim(\C)$ if and only if $\C\simeq\I\boxtimes\C(C_m,q)$ for an Ising modular fusion category $\I$ and odd $m$.
\end{enumerate}
\end{corollary}

\begin{proof}
Nilpotent fusion categories are weakly integral (by induction on the nilpotency class, as all components of a faithful grading have the same Frobenius--Perron dimension), hence pseudounitary, and each $\C_p$ is a pseudounitary modular fusion category with $\dim(\C_p)=\FPdim(\C_p)=\Ndim(\C_p)$. Since the $\FSexp(\C_p)$ and $\dim(\C_p)$ are powers of distinct primes, $\FSexp(\C)=\prod_p\FSexp(\C_p)$ divides $\dim(\C)$, resp.\ $4\dim(\C)$, if and only if $\FSexp(\C_p)$ divides $\dim(\C_p)$ for all $p$, resp.\ for all odd $p$ with $\FSexp(\C_2)\mid4\dim(\C_2)$. For odd $p$ this holds by Theorem \ref{thm:main}\ref{main:odd}, and $\FSexp(\C_2)\mid4\dim(\C_2)$ by Theorem \ref{thm:main}\ref{main:even}. Moreover $\FSexp(\C_2)\nmid\dim(\C_2)$ exactly when $\FSexp(\C_2)>\dim(\C_2)$, which gives \ref{nilp:div} by Lemma \ref{lem:evenlist}. Finally, $\FSexp(\C)=4\dim(\C)$ if and only if $\FSexp(\C_2)=4\dim(\C_2)$ and $\FSexp(\C_p)=\dim(\C_p)$ for all odd $p$, i.e.\ $\C_2$ is Ising by Lemma \ref{lem:evenlist} and $\C_p\simeq\C(C_{p^{n_p}},q_p)$ by Corollary \ref{cor:modular}; conversely, a nondegenerate quadratic form on a cyclic group of odd order $m$ takes values of order exactly $m$.
\end{proof}

\subsection*{Acknowledgments}
The second author is grateful for the hospitality and excellent working conditions at the Department of Mathematics at Universit\"at Hamburg where she carried out part of this research as an Experienced Fellow of the Alexander von Humboldt Foundation. During a revision of the completed manuscript, Claude (Anthropic) was used to obtain referee-style feedback, to suggest corrections and condensations of arguments, and to draft proposed proofs and additions, some of which were incorporated after revision by the authors. The authors independently evaluated all suggestions and verified all resulting changes. All mathematical statements, proofs, computational results, and references were independently checked by the authors, who take sole responsibility for the contents of the paper.

\end{document}